\documentclass{amsproc}
\usepackage{graphicx}
\usepackage{amssymb}
\usepackage{epstopdf}
\usepackage{amsmath,amsthm,amscd,amssymb}
\usepackage{latexsym}
\usepackage[colorlinks,citecolor=red,pagebackref,hypertexnames=false]{hyperref}
\usepackage{geometry}                
\numberwithin{equation}{section}

\theoremstyle{plain}
\newtheorem{theorem}{Theorem}[section]
\newtheorem{lemma}[theorem]{Lemma}
\newtheorem{corollary}[theorem]{Corollary}
\newtheorem{proposition}[theorem]{Proposition}

\theoremstyle{definition}
\newtheorem{definition}[theorem]{Definition}

\newtheorem{example}[theorem]{Example}

\newtheorem{case[theorem]}{Case}

\theoremstyle{remark}

\numberwithin{equation}{section}

\begin{document}

\title{A multi-parameter variant of the Erd\H os distance problem} 

\author{A. Iosevich, M. Janczak and J. Passant}

\date{today}

\email{iosevich@math.rochester.edu}
\email{mjanczak@u.rochester.edu}
\email{jpassant@ur.rochester.edu}

\address{Department of Mathematics, University of Rochester, Rochester, NY 14627}

\thanks{This work was partially supported by the NSA Grant H98230-15-1-0319}

\begin{abstract} 
We study the following variant of the Erd\H os distance problem. Given $E$ and $F$ a point sets in $\mathbb{R}^d$ and $p = (p_1, \ldots, p_q)$ with $p_1+ \cdots + p_q = d$ is an increasing partition of $d$ define
$$ B_p(E,F)=\{(|x_1-y_1|, \ldots, |x_q-y_q|): x \in E, y \in F \},$$  
where $x=(x_1, \ldots, x_q)$ with $x_i$ in $\mathbb{R}^{p_i}$. For $p_1 \geq 2$ it is not difficult to construct $E$ and $F$ such that $|B_{p}(E,F)|=1$. On the other hand, it is easy to see that if $\gamma_q$ is the best know exponent for the distance problem in $\mathbb{R}^{p_i}$ that $|B_p(E,E)| \geq C{|E|}^{\frac{\gamma_q}{q}}$. The question we study is whether we can improve the exponent $\frac{\gamma_q}{q}$.

We first study partitions of length two in detail and prove the optimal result (up to logarithms) that
$$ |B_{2,2}(E)| \gtrapprox |E| $$
In the generalised two dimensional case for $B_{k,l}$ we need the stronger condition that $E$ is $s$-adaptable (\cite{IRU14}) for $s<\frac{k}{2}+\frac{1}{3}$, letting $\gamma_m$ be the best known exponent for the Erd\H os-distance problem in $\mathbb{R}^m$ for $k \neq l$ we gain a further optimal result of,
$$ |B_{k,l}(E)| \gtrapprox |E|^{\gamma_l}.$$
When $k=l$ we use the explicit $\gamma_m=\frac{m}{2}-\frac{2}{m(m+2)}$ result due to Solymosi and Vu (\cite{SV08}) to gain
$$ |B_{k,k}(E)| \gtrapprox |E|^{\frac{13}{14}\gamma_k}.$$
For a general partition, let $\gamma_i = \frac{2}{p_i}-\frac{2}{p_i(p_i+2)}$ and $\eta_i = \frac{2}{2d-(p_i-1)}$. Then if $E$ is $s$-adaptable with $s>d-\frac{p_1}{2}+\frac{1}{3}$ we have
$$ B_p(E) \gtrapprox |E|^\tau \hspace{0.5cm} \text{where} \hspace{0.5cm} \tau = \gamma_q\left(\frac{\gamma_1+\eta_1}{\gamma_q+(q-1)(\gamma_1+\eta_1)}\right).$$
Where $p_i \sim \frac{d}{q}$ implies $\tau \sim \gamma_{q}\left(\frac{1}{q}+\frac{1}{dq}\right)$ and $p_q \sim d$ (with $q<<d$) implies $\tau \sim \gamma_{q}\left(\frac{1}{q}+\frac{1}{q^2}\right)$.
\end{abstract} 

\maketitle


\section{Introduction}

\vskip.125in 

Given a set $E$ in $\mathbb{R}^d$, the distance set of $E$ is
$$\Delta_d(E) = \{|x - y| : x, y \in E\} \subseteq \mathbb{R}.$$
In \cite{E45} Erd\H os posed the question: What is the minimal number of distinct distances determined by a finite point set $E$ in $\mathbb{R}^d$?
This has been thoroughly studied in both the $d=2$ case where the cascade of improvements to Erd\H os original $|E|^\frac{1}{2}$ by authors including Moser \cite{M52}, Chung \cite{C84}, Chung-Szemer\' edi-Trotter \cite{CST92}, Sz\' ekely \cite{S97}, Solymosi-T\' oth \cite{ST01}, Tardos \cite{T03} and most recently the solution of the problem in two dimensions due to Guth-Katz \cite{GK15}. In higher dimensions a simple variant of Erd\H os original argument gives $|E|^\frac{1}{d}$ in dimension $d$. An improvement in three dimensions due Clarkson-Edelsbrunner-Gubias-Sharir-Welzl \cite{CEGSW90} proved that one obtains at least $|E|^{\frac{1}{2}}$ distances, the three dimentional bound was furthered by Aronov-Pach-Sharir-Tardos \cite{APST04} who also proved a small improvement over the $|E|^{\frac{1}{d}}$ bound in dimension $d$. This was then improved significantly by Solymosi-Vu \cite{SV08} who proved one obtains at least $|E|^{\frac{2}{d} - \frac{2}{d(d+2)}}$ distances, a near optimal bound for large dimensions.

Recently Birklbauer-Iosevich \cite{BI17} and Hambrook-Iosevich-Rice \cite{HIR17} introduced a higher parameter variant of the Erd\H os distance problem in the contexts of finite fields and analytic setting respectively. In this paper, we study the following real variant of this high parameter Erd\H os-distance problem. 

Let $d=k+l$ and $x=(x_1,x_2)$, where $x_1$ is the vector of the first $k$ coordinates of $x$ and $x_2$ the vector of the final $l$ coordinates. Given $E \subset {\Bbb R}^d$, $d \ge 3$ and $2 \leq k \leq l$ we define
$$ B_{k,l}(E,F)=\{(|x_1-y_1|,|x_2-y_2|): x \in E, y \in F\},$$
the case $k=1$ is less interesting because in that case $x_1-y_1$ is a one-dimensional quantity. By taking $E \subset S^{k} \times \{(0, \ldots, 0)\}$ and $F \subset \{0, \dots, 0\} \times S^{l}$, we obtain $|B(E,F)|=1$.

Due to examples of this type, we will either have to impose stricter conditions on our $E$ and $F$ or consider the case were $E=F$. For most of this paper we discuss the latter of these options and denote $B_{k,l}(E) := B_{k,l}(E,E)$. In this case we can use distinct distance bound in projections to gain `trivial' lower bounds on such sets. Suppose that $E$ is a point set in $\mathbb{R}^d$ for $d=k+l$ we must have that either the projection onto the first $k$ coordinates or the projection onto the final $l$ coordinates has at least $n^{\frac{1}{2}}$ points in it. We let $\gamma_m$ be the best exponent for the Erd\H os-distance problem in $\mathbb{R}^m$, using the above we gain the `trivial' bound

\begin{equation}\label{ETrivialkl}
|B_{k,l}(E)| \gtrsim |E|^{\frac{\gamma_{l}}{2}}.
\end{equation}

We can also introduce this problem in the wider context by considering diving $d$ up into a partition of length more than two. Here let $P_q(d)$ be the set of increasing partitions of the positive integer $d$ consisting of $q$ elements, denote the members of $p$ in $P_q(d)$ as $(p_1, \ldots, p _q)$, thus we have $p_1+\cdots +p_q=d$ and $p_1 \leq \cdots \leq p_q$. Suppose that $E$ is a finite point set in $\mathbb{R}^d = \mathbb{R}^{p_1}\times \cdots \times \mathbb{R}^{p_q}$, for $x$ in $E$ we let $x_i$ to be the projection of $x$ into $\mathbb{R}^{p_i}$. Then we define

$$ B_{p}(E,F) = \{ (|x_1-y_1|, \ldots, |x_q-y_q|) :  x \in E, y \in F\}$$

We note that the $q=1$ case corresponds to the usual distinct distance conjectures thus we will consider $q\geq 2$, we also ignore cases where $p_1 = 1$ as then we again have that $x_1-y_1$ is a one-dimensional quantity. Similar to above by taking $E \subset S^{p_1} \times \{(0, \ldots, 0)\}$ and $F \subset \{0, \dots, 0\} \times S^{p_q}$, we obtain $|B(E,F)|=1$. This example only exploits the first and last projections, but one could modify $E$ so that it has alternating spheres and zeros for each $p_i$ and $F$ constructed with the opposite order, this would then exploit all projections. 

We again consider the case were $E=F$ and denote $B_p(E) := B_p(E,E)$. To gain the `trivial' bound we observe that one projection has at least $n^{\frac{1}{q}}$ points in it so we have

\begin{equation}\label{ETrivialGenp}
|B_p(E)| \gtrsim |E|^{\frac{\gamma_{p_q}}{q}}.
\end{equation}

We will later see an important use for $B_p(E,F)$ in the situation where we have that $E \neq F$ in general. In this case we impose the condition of $s$-adaptability, which gives the set a sufficient separability that such counter examples as above cannot be constructed.


Our goal in this paper is to beat the estimate (\ref{ETrivialGenp}). We begin with the case of partitions of length $2$ and later consider the larger parameter variants of this problem.

~\\
\vskip.125in 

\section{Statement of results} 

\vskip.125in 

We now proceed to state the main results of this paper. \\

\vskip.125in

\subsection{Partitions of Length Two}~

\vskip.125in 

First we consider the simplest case; that of a partition of length two. Let $p=(k,l)$ thus $d=k+l$ with $k\leq l$, given $E \subset {\Bbb R}^d$ we denote $B_p(E)$ by $B_{k,l}(E)$. It is easy to see that 
\begin{equation} \label{projectiontrivial} |B_{k,l}(E)| \ge \max \{|\Delta(\pi_1(E))|, |\Delta(\pi_2(E))| \}, \end{equation} 
where $\pi_1(x)=x_1$ the projection onto the first $k$ coordinates and $\pi_2(x)=x_2$ the projection onto the final $l$ coordinates. Since at least one of $\pi_1(E), \pi_2(E)$ has size $\ge {|E|}^{\frac{1}{2}}$. It is clear to see that at least one of the projections must contain at least $|E|^{\frac{1}{2}}$ of our points and using the distinct distance bound obtained by Solymosi and Vu \cite{SV08} we obtain the following bound
\begin{equation} \label{trivial} | B_{k,l}(E)| \gtrapprox |E|^{\frac{1}{l}-\frac{1}{l(l+2)}}.\end{equation} 

Our goal is to beat the estimate (\ref{trivial}). 
~\\
\vskip.125in

\subsection*{The Case $k=2$, $l=2$} ~

\vskip.125in 

Suppose that $k=2, l=2$. We can use the Guth-Katz solution of the Erd\H os-distance conjecture (\cite{GK15}) to gain the `trivial' bound

\begin{equation}
|B_{2,2}(E)| \gtrapprox |E|^{\frac{1}{2}}. \label{eq: trivial2,2}
\end{equation}

Our aim is to beat this estimate. Let $\nu(t_1,t_2)$ be the number of repetitions of the distance pair $(t_1,t_2)$ in $B(E)$. Using the known bounds for the single distance conjecture (\cite{SST84}), 
$$\nu(t_1,t_2) \lesssim {(|\pi_1(E)| \cdot |\pi_2(E)|)}^{\frac{4}{3}},$$ we deduce that

$$ |B_{2,2}(E)| \gtrsim  \frac{{|E|}^2}{{(|\pi_1(E)| \cdot |\pi_2(E)|)}^{\frac{4}{3}}}.$$
Combined with (\ref{projectiontrivial}), we would obtain
$$ |B_{2,2}(E)| \gtrsim {|E|}^{\frac{6}{11}}.$$ 
Assuming the single distance conjecture, $\nu(t_1,t_2) \lessapprox |\pi_1(E)| \cdot |\pi_2(E)|$, where here and throughout, $X \lessapprox Y$ with the controlling parameter $R$ means that for every $\epsilon>0$ there exists $C_{\epsilon}>0$ such that $X \leq C_{\epsilon}R^{\epsilon}Y$. It would follow that 
$$ |B_{2,2}(E)| \gtrapprox \frac{{|E|}^2}{|\pi_1(E)| \cdot |\pi_2(E)|}.$$
Combining this with (\ref{projectiontrivial}) once again yields the following result. 

\begin{theorem} Let $E \subset {\Bbb R}^4$. Then 
\begin{equation} \label{Linfty} |B_{2,2}(E)| \gtrsim {|E|}^{\frac{6}{11}}. \end{equation} 

If we assume the Erd\H os single distance conjecture, 
\begin{equation} \label{Linftydream} |B_{2,2}(E)| \gtrapprox {|E|}^{\frac{2}{3}}. \end{equation} 
\end{theorem} 

%

Using dyadic pigeonholing we can improve this to a sharp bound.

\begin{theorem}\label{TB22Sharp} Suppose that $E$ is a finite point set in $\mathbb{R}^4$ then
$$ |B_{2,2}(E)| \gtrapprox |E|.$$
\end{theorem} 
\vskip.125in 

\subsection*{The case $k=2, l=3$}~

\vskip.125in 

First note that the trivial bound we are trying to beat in this case is $|B_{2,3}(E)| \gtrapprox |E|^{\frac{3}{10}}$. In this case we can again used improved estimates on unit distances from \cite{CEGSW90} to gain 

$$|B_{2,3}(E)| \gtrapprox \frac{|E|^2}{|\pi_1(E)|^{\frac{4}{3}} |\pi_2(E)|^{\frac{3}{2}}} \ge \frac{|E|^2}{A^{\frac{17}{6}}},$$

where $A=\max\{|\pi_1(E)|,|\pi_2(E)|\}$. We note that using the bounds obtained for distinct distances in three dimensions from \cite{SV08} we have the `trivial' bound of

$$ |B_{2,3}(E)| \gtrapprox \max\{|\pi_1(E)|, |\pi_2(E)|^{\frac{3}{5}}\} \ge A^{\frac{3}{5}}$$.

Combining these one gets the first improvement

$$|B_{2,3}(E)|\gtrapprox |E|^{\frac{36}{103}}=|E|^{0.3495\ldots}.$$

Using the same idea as in the $(2,2)$ case we gain the following bound


Again using the same pigeonhole technique in the $(2,2)$ case we have the following optimal bound

\begin{theorem}\label{T23Optimal} Suppose $E$ is a finite point set in $\mathbb{R}^5$. If $\gamma_3$ is the best known exponent for the Erd\H os-distance problem in $\mathbb{R}^3$, then we have that
$$ |B_{2,3}(E)| \gtrsim |E|^{\gamma_3}.$$
\end{theorem}

Using the Guth-Katz solution of the Erd\H os-distance problem in the plane (\cite{GK15}) and the iterative argument of Solymosi-Vu (\cite{SV08}) we have $\gamma_3 \geq \frac{3}{5}$, thus
$$ |B_{2,3}(E)| \gtrapprox |E|^{\frac{3}{5}}.$$

\vskip.125in 

\subsection*{Arbitrary $(k,l)$} ~

\vskip.125in 

In the $(2,2)$ case a crucial aspect of the proof is to take two distinct sets and gain a bound on the distinct distances between them. As we saw above, this method will not work in general when one of our dimensions $k$ or $l$ is larger than $4$. Let  $\Delta_m(A,B) = \{|x-y| : x \in A, y \in B \}$, be the generalised distance set in dimension $m$. Supposing we have sets $A$ and $B$ of the same size in $\mathbb{R}^m$ satisfying some condition $C_m$ dependent on a dimension $m$ such that we can extract a useful exponent for the distance set i.e. there exists $\delta_m$ such that
\begin{equation} \label{E: GenericPairSetDistanceBound}
|\Delta_m(A,B)| \gtrapprox |A|^{\delta_m}
\end{equation}
We can then prove the following theorem.
\begin{theorem}\label{TBklNearOptimal} Let $E$ be a point set in $\mathbb{R}^{d}$ with $d=k+l$, let $\gamma_m$ be the best bound for distinct distances in dimension $m$ and $\delta_m$ the best exponent for distances between two different sets of the same size in dimension $m$ under some condition $C_m$. If $E$ satisfies $C_k$ and $C_l$ and we have that $\delta_{m+1}\geq \gamma_m$then we have two results. First if $k\neq l$ we have that
$$ |B_{k,l}(E)| \gtrapprox |E|^{\min \{\gamma_l, \delta_k \}}. $$
If $k=l$ we suppose that $k,l \geq 3$ (as the case $(2,2)$ is studied above in greater detail) suppose that $\zeta\gamma_k =  \delta_k$ where it is clear that $\zeta \leq 1$, then
$$ |B_{k,k}(E)| \gtrapprox |E|^{\zeta \gamma_k}.$$
\end{theorem}

The condition we will impose on our point set $E$ will be to ensure its points are sufficiently separated. Using the mechanism introduced in \cite{IRU14} and \cite{IL05}, one can use a thickening of our point set to allow us to use bounds gained for the Falconer Conjecture. We can then bring these bounds back into our discrete setting. 

An important consideration in this thickening is for our original distances to contribute and even density to the measure of our continuous distance set. If we have a point set too clustered then overlaps in the thickening would not allow us to bring our results back to the discrete setting. To avoid this we introduce the notion of $s$-adaptability of a point set. We define this in terms of the $s$-dimensional energy of the thickening of our set.

We start with a point set $E$ of size $n$ in $[0,1]^d$, if $E$ was not already in the unit cube we first rescale $E$ by dividing by its diameter then translate it so that it is a subset of the unit cube. We want to then thicken $E$ by a certain amount dependent on the number of points in $E$. We choose an arbitrary parameter $s$, whose range will be determined later, and consider thickening our discrete set $E$ by forming balls of size $n^{-\frac{1}{s}}$ around each point. We now define an `indicator' measure on $E$.

\begin{definition} (\cite{IRU14}) Let $E$ be a set of $n$ points contained in $[0,1]^d$, we define the measure
$$ d\mu^s_E(x) = n^{-1}\cdot n^{\frac{d}{s}} \cdot \sum_{e\in E} \varphi(n^{\frac{1}{s}}(x-e))dx, $$
where $\varphi$ is a $C^{\infty}$ bump function with support on $[-1,1]^d$. Note that this is not quite a probability measure, thought it is clearly finite and bounded by some $C$ in $\mathbb{R}$ independent of $n$ and $s$.
\end{definition}

To capture the nation that the points in $E$ are not too clustered we look at continuous energies.

\begin{definition}
Suppose that $\mu$ is a measure on $\mathbb{R}^m$ we define its $s$-dimensional energy to be
$$ I_s (\mu) := \int \int |x-y|^{-s} d\mu(x)d\mu(y) = C_{s,m} \int | \widehat{\mu} (\xi) |^2 |\xi|^{s-m} d\xi. $$
\end{definition}
 
We can now define the notion of $s$-adaptability as 
 
\begin{definition} A point set $E$ in $[0,1]^d$ is $s$-adaptable if $s$-dimensional energy $I_s(\mu_E^s)$ is finite. That is
$$ I_s(\mu_E) = \int \int |x-y|^{-s} d\mu^s_E(x) d\mu_E^s(y) < \infty.$$
\end{definition}

To see how this applies directly to the discrete set $E$ we provide the following equivalent definition.

\begin{proposition}\label{PContToDiscEnergy} A point set $E$ in $[0,1]^d$ is $s$-adaptable if
\begin{equation}\label{EqFinEnergy}
n^{-2}\sum_{e\neq e'} |e-e'|^{-s} \lesssim 1.
\end{equation} 
\end{proposition}

To see how $s$-adaptability captures the notion that a set $E$ is not too clustered, suppose $E$ comes from a $1$-separated set scaled down by its diameter. Then \ref{EqFinEnergy} can be written,
$$ n^{-2}\sum_{e\neq e'} |e-e'|^{-s} \lesssim (\text{diameter}(E))^{-s}. $$
Thus an $s$-adaptable set is a scaled $1$-separated set in which the average distance between two points raised to the power of $-s$ is comparable to the diameter of the set. Unfortunately not all sets are $s$-adaptable as seen in \cite{IRU14}, however many natural sets are notably the homogeneous sets studied by Solymosi and Vu \cite{SV04}. With this notion of $s$-adaptability in place we gain the following near optimal result for pair distances.

%
%

\begin{theorem}\label{TPairDistFalconerDiscrete}
Suppose that $A$ and $B$ are two $s$-adaptable subsets of $\mathbb{R}^m$ with $s> \frac{m}{2}+\frac{1}{3}$ and $|A|=|B|$. Let $\Delta_m(A,B) = \{|x-y| : x \in A, y \in B \}$ be the set of $m$-dimension distances between $A$ and $B$, then we have
$$ |\Delta_m(A,B)| \gtrsim  |A|^{\frac{2}{m+1}}.$$
\end{theorem}

This allows us to prove the following Corollary of Theorem \ref{TBklNearOptimal}.

\begin{corollary} \label{Coro: Bkl}
Let $E$ be a point set in $\mathbb{R}^{d}$ with $d=k+l$.
If $s>\frac{l}{2}+\frac{1}{3} \geq \frac{k}{2}+\frac{1}{3}$ and $E$ is $s$-adaptable we have two cases. First if $k\neq l$ we have that
$$ |B_{k,l}(E)| \gtrapprox |E|^{\gamma_l}. $$
If $k=l$ we suppose that $k,l \geq 3$ (as the case $(2,2)$ is studied above in greater detail) and that $\zeta\gamma_k =  \delta_k$ where it is clear that $\zeta \leq 1$, then
$$ |B_{k,k}(E)| \gtrapprox |E|^{\frac{13}{14} \gamma_k}.$$
\end{corollary}

\vskip.125in

\subsection{Larger Parameter Variants}

\vskip.125in

At this stage we study the problem in its full generality. We show that the above method can be applied here as well, although our bounds achieved are far from optimal.

%
%

\begin{equation}
|B_p(E)| \gtrsim |E|^{\frac{\gamma_{p_q}}{q}}.
\end{equation}

In the case where $p$ is a partition of $d=2m$ into all twos, we call $p$ a partition of twos and note that all projections are into $\mathbb{R}^2$. Here we modify our notation; for a a point set $A$ in $\mathbb{R}^{2m}$ let $B_2^m(A)$ denote the set $B_p(A)$ where $p$ is the partition of $2m$ containing all twos. Thus we can use the Guth-Katz solution to the Erd\H os distance problem to gain the `trivial' bound

\begin{equation} \label{E2qTimes}
 |B_2^q(E)| \gtrapprox |E|^{\frac{1}{q}}.
\end{equation}\\

\vskip.125in 

\subsection*{Partitions of Twos}~\\

The first interesting case not already studied of partitions of twos is $B_2^3(E)$ for $E$ in $\mathbb{R}^6$. Unfortunately the pideonholing argument deployed above does not lead to improvements on the above trivial bound at present we will present the best know bounds due to this method when considering the most general partitions in Theorem \ref{TGenPartsDydPideon}. Fortunately, we can employ a method based on more concrete forms of density averaging method to gain the following

\begin{theorem}\label{T222}
Let $s$ be the minimal dimension such that two continuous sets $A$ and $B$ of dimension $>s$ have that $\mathcal{L}(B_{2,2}(A,B))>0$. Suppose that $E$ is an $s$-adaptable point set in $\mathbb{R}^6$ with $\eta_1$ the best exponent for $B_{2,2}(A,B)$ for two discrete $s$-adaptable sets $A$ and $B$ in $\mathbb{R}^4$. Then
$$ |B_2^3(E)| \gtrapprox |E|^\frac{1+\eta_1}{3+2\eta_1}.$$
\end{theorem}

In this proof, the bound we obtained for $B_{2,2}(A,B)$ plays a crucial role. Extracting this we gain an iterative result, 

\begin{theorem}\label{T2qTimes}
Suppose that $E$ is an $s$-adaptable point set in $\mathbb{R}^{2q}$ then if $\eta_m$ is the best known exponent for $B_2^m(A,B)$ in dimension $2m$ for two $s$-adaptable point sets $A$ and $B$ in $\mathbb{R}^{2m}$ we have
$$ |B_2^q(E)| \gtrapprox |E|^\theta \hspace{0.5cm} \text{where} \hspace{0.5cm} \theta = \frac{1+\eta_{q-1}}{q+(q-1)\eta_{q-1}} = \frac{1}{q} + \frac{\eta_{q-1}}{q(q+(q-1)\eta_{q-1})}.$$
\end{theorem}

We can use the following result of Hambrook, Iosevich and Rice (\cite{HIR17}) to extract a useful exponent.

\begin{theorem}[\cite{HIR17}]\label{TPartofTwosCont} Suppose that $A$ and $B$ are two sets in $[0,1]^d$ both of size $n$. Suppose that $d=2q$. Then if $A$ and $B$ are $s$-adaptable with $s>d-\frac{2}{3}=2\left( q - \frac{1}{3} \right)$ then,

$$ \mathcal{L}(B_2^q(A,B)) > 0. $$
\end{theorem}

This allows us to take $\eta_{q-1}=\frac{2}{4(q-1)}$ (technically one could use $\frac{2}{4q-5}$ but this first is more computationally convenient and there is no real difference) and thus we have that

\begin{corollary}
Suppose that $E$ a finite point set in $\mathbb{R}^d$ is $(d-\frac{5}{3})$-separable set. Then
$$ |B_2^q(E)| \gtrapprox |E|^{\theta_0} \hspace{0.5cm} \text{where} \hspace{0.5cm} \theta_0=\frac{1}{q} + \frac{2}{q(4q+1)(q-1)} \sim \frac{1}{q} + \frac{1}{q^3}.$$
\end{corollary}

\vskip.125in 

\subsection*{General Partitions}~\\

For the following result it is useful to discuss partitions $p$ of $d$ with a certain element, say $p_i$, removed. For this we will use the notation $p \setminus p_i$ which is then an increasing partition of $d-p_i$. For general partitions can again apply the methods from our $(2,2)$-case to gain the following bound

\begin{theorem}\label{TGenPartsDydPideon}Suppose that $E$ is a point set of size $n$ in $\mathbb{R^d}$, with $p$ be an increasing partition of $d$. Suppose that $s_i$ is the best lower bound on the dimensions of two sets $A$ and $B$ in $\mathbb{R}^{d-p_i}$ such that $\mathcal{L}(B_{p\setminus p_i}(A,B)) > 0$, with $\eta_i < \frac{1}{s_i}$. Let $\gamma_i$ be the best exponent of the Erd\H os-distance problem is $\mathbb{R}^{p_i}$. Suppose that $E$ is $s$-adaptable for $s>d_i$ for all such $i$, then we have that

$$ B_p(E) \gtrapprox |E|^{\min\{\gamma_i, \eta_i\}}$$
\end{theorem}

The best know bound for the dimension $s_i$ above is given by the following result due to Hambrook, Iosevich and Rice.

\begin{theorem}[\cite{HIR17}]\label{TGenPartCont} Suppose that $A$ and $B$ are two sets in $[0,1]^d$ both of size $n$. Suppose that $p$ is an increasing partition of $d$. Then if $A$ and $B$ are $s$-adaptable with $s>d- \frac{p_i}{2}+\frac{1}{3}$ for all $p_i$ then,

$$ \mathcal{L}(B_p(A,B)) > 0. $$
\end{theorem}

We not that this fails to improve over the bound obtained in (\ref{ETrivialGenp}), however we can apply a different method to gain a slight improvement. We use the compilation between one rich projection and lost of sparser projections. As above, we are left with an estimate for a $B_{p\setminus p_i}$ for two distinct sets, here we again appeal to the continuous case to gain a bound.

\begin{theorem}[\cite{HIR17}]\label{TGenPartCont} Suppose that $A$ and $B$ are two sets in $[0,1]^d$ both of size $n$. Suppose that $p$ is an increasing partition of $d$. Then if $A$ and $B$ are $s$-adaptable with $s>d- \frac{p_i}{2}+\frac{1}{3}$ for all $p_i$ then,

$$ \mathcal{L}(B_p(A,B)) > 0. $$
\end{theorem}

\begin{theorem}\label{TGenPart}Suppose that $E$ is a point set of size $n$ in $\mathbb{R^d}$, with $p$ be an increasing partition of $d$. Suppose that $s_i$ is the best lower bound on the dimensions of two sets $A$ and $B$ in $\mathbb{R}^{d-p_i}$ such that $\mathcal{L}(B_{p\setminus p_i}(A,B)) > 0$, with $\eta_i < \frac{1}{s_i}$. Let $\gamma_i$ be the best exponent of the Erd\H os-distance problem is $\mathbb{R}^{p_i}$. Suppose that $E$ is $s$-adaptable for $s>d_i$ for all such $i$, then we have that

$$ B_p(E) \gtrapprox |E|^\tau \hspace{0.5cm} \text{where} \hspace{0.5cm} \tau = \gamma_q\left(\frac{\gamma_i+\eta_i}{\gamma_q+(q-1)(\gamma_i+\eta_i)}\right).$$
\end{theorem}

Using Theorem \ref{TGenPart} along with both Theorem \ref{TGenPartCont} and the exponent $\gamma_i = \frac{2}{p_i}-\frac{2}{p_i(p_i+2)}$ form \cite{SV08} we have the following.


\begin{corollary}Suppose that $E$ is a point set of size $n$ in $\mathbb{R^d}$, with $p$ be an increasing partition of $d$. Let $\gamma_i = \frac{2}{p_i}-\frac{2}{p_i(p_i+2)}$ and $\eta_i = \frac{2}{2d-(p_i-1)}$. Suppose that $E$ is $s$-adaptable for $s>d-\frac{d_i}{2}+\frac{1}{3}$ for all such $i$, then we have that

$$ B_p(E) \gtrapprox |E|^\tau \hspace{0.5cm} \text{where} \hspace{0.5cm} \tau = \gamma_q\left(\frac{\gamma_1+\eta_1}{\gamma_q+(q-1)(\gamma_1+\eta_1)}\right).$$
\end{corollary}

The above result follows from simple calculations that shows that with our particular values of $\gamma_i$ and $\eta_i$ we have $\gamma_i+\eta_i$ is maximised when $i=1$. One can show that $\frac{\gamma_q}{\gamma_1+\eta_1} \sim \frac{p_1}{p_q}$ asymptotically as $d$ grows. Thus, $\tau \sim \gamma_q\left(\frac{1}{q} + \frac{p_q-p_1}{(p_1+p_q(q-1))q)}\right)$. When one has the all the $p_i$ are comparable with $\frac{d}{q}$, we have $p_q-p_1 \sim 1$ and thus $\tau \sim \gamma_{q}\left(\frac{1}{q}+\frac{1}{dq}\right)$. When one has that $p_q \sim d$ (with $q<<d$) we have that $\tau \sim \gamma_{q}\left(\frac{1}{q}+\frac{1}{q^2}\right)$, which gives us a better estimate than in the case above where all our values $p_i$ approximately equal.

\vskip.125in 

\section{Proof of Theorem \ref{TB22Sharp}} 

\vskip.125in 

First we prove a Lemma for distances between two distance sets in $\mathbb{R}^2$ using the methods of Elekes-Sharir and Guth-Katz.

\begin{lemma}\label{G-KTwoSets} Let $E$ and $F$ be two point sets in $\mathbb{R}^2$, both of size $n$. We define the distance set between $E$ and $F$ as,
$$ \Delta(E,F) = \{ |x-y| : x\in E, y \in F \}. $$
Then $|\Delta(E,F)| \gtrapprox n$.
\begin{proof}This is a quick Corollary of the Guth-Katz proof of the Erd\H os distance problem in the plane \cite{GK15}, using the Elekes-Sharir Framework. Let us represent our finite set of distances as $\Delta(E,F)=\{ d_1, \ldots, d_k \}$ and define $E_i = \{ (x,y) \in E \times F : |x-y| = d_i\}$. We consider the set of quadruples,
$$Q(E,F) \{ (x,y,x',y') \in E\times F\times E \times F : |x-y|=|x' - y'|\}.$$
We bound this from above by noting that $Q(A,B) \subseteq Q(A\cup B)$, where this latter set is the quadruples where each element comes from $A\cup B$. Then by Guth-Katz $|Q(A, B)| \leq |Q(A\cup B)|\lesssim n^3\log(n)$. To bound $|Q(A, B)|$ from below one uses Cauchy-Schwarz,
\begin{align*}
|Q(A, B)| &= \sum_{i=1}^{|Q(A, B)|} 2{{|E_i|}\choose{2}} \\
						  &\gtrsim \sum_i (|E_i|)^2 \\
						  &\geq \frac{1}{|(\Delta(E,F)|}\left(\sum_i |E_i| \right)^2 \\
						  &\gtrsim \frac{n^4}{|\Delta(E,F)|}.						  
\end{align*}
Combining these gives, $|\Delta(E,F)| \gtrsim \frac{n}{log(n)}$, which implies our result.
\end{proof}
\end{lemma}

This argument uses dyadic pigeonholing, which we can represent as the following Lemma. 

\begin{lemma}\label{DyadicPH}
Suppose we have $N$ objects, say $A_1, \ldots, A_N$. Each of these objects has a real number associated with it, say $a_i$ for $i=1, \ldots, N$, where $a_i = O(n^{O(1)})$. Then there is a collection of $\sim \frac{N}{\log(n)}$ of our original objects where for those objects we have that $a_i \in [c, 2c]$ for some $c$ in $\mathbb{R}$. Thus we can say $a_i \sim c$.
\begin{proof}
As we have that $a_i = O(n^{O(1)})$ we can divide our $a_i$ into $\sim \log(n)$ sets, where the $j^{th}$ such set contains those elements $a_i$ such that $2^j \leq a_i \leq 2^{j+1}$. We have $N$ total numbers and $\log n$ sets for them to be placed in, so by the pigeonhole principle we have that one such set must contain at least $\frac{N}{\log n}$ such points, we throw away all others. By our method of separation we have that our remaining values of $a_i$ are in an interval $[c, 2c]$ and we have at least $\frac{N}{\log n}$ of our original points remaining.
\end{proof}
\end{lemma}

Recall that we denote the projection of $E$ onto its first two coordinates by $\pi_1(E) = \pi_1$, and for the projection of $E$ onto its second two coordinates we use $\pi_2(E) = \pi_2$. We now use dyadic pigeonholing to separate the points in $E$ based on there `richness' in the second plane. For each point $x$ in $E$ we specify a $\pi_2$-richness $R(x) = |\pi_2(\pi_1^{-1}(\{ \pi_1(x) \}))|$. Note that this takes a point $x$ in $E$ and associates to it the number of predecessors $\pi_1(x)$ has in $E$, it also gives all points that map down to one value in $\pi_1$ exactly the same richness. We note that as $|\pi_2| \leq |E|$ we have that $|R(x)| \leq |E|$.

We are now ready to apply the dyadic pigeonholing. By Lemma \ref{DyadicPH} we thus have $\frac{E}{\log(|E|)}$ members of our original $E$ that have some fixed $\pi_2$-richness, say $r$. For emphasis, we call this new subset of $E$ with fixed richness $E'$ were $|E'|=\frac{|E|}{\log(|E|)} \approx |E|$. We note that as our function $R(x)$ is constant on all points mapping to the same point under our $\pi_1$ mapping we do not lose any such points under our pigeonholing, thus each point in $E'$ is one of $\sim r$ that sit above a given member of $\pi_1(E')$ and each point of  $\pi_1(E')$ has $\sim r$ points above it. This means that $|E'| \sim |\pi_1(|E'|)|r$ but as $|E'|\approx |E|$ we have that $|E| \approx |\pi_1(E')|r$.

We now count the number of distance pairs in $B_{2,2}(E')$. Consider that we have $|\pi_1(E')|$ points in the first projection which is now a subset of $\mathbb{R}^2$ and thus we can use Guth-Katz bound for the Erd\H os-distance problem in the plane to obtain $\gtrapprox |\pi_1|$ distances. We fix one such distance and look at a fixed pair of endpoints. By our earlier analysis these points have $\sim r$ predecessors and so we can project these predecessors into our second projection to obtain two different sets of size $\sim r$. Then by Lemma \ref{G-KTwoSets} these two sets generate $\gtrsim r$ distances. As this is for a fixed one of our $\gtrsim |\pi_1|$ distances we gain $\gtrsim r$ distances in the second coordinate for each of these and thus we have $ \gtrapprox |\pi_1|r $ distance pairs in $B_{2,2}(E')$ in general. Thus we can combine the above to gain 
$$B_{2,2}(E) \geq B_{2,2}(E') \gtrapprox  |\pi_1|r \approx |E|.$$

\vskip.125in

\section{Proof of Theorem \ref{TBklNearOptimal}} 

\vskip.125in

As we saw in the proof of Theorem \ref{TB22Sharp} the distances between two different sets plays a major role. We again use dyadic pigeonholing where we use the conditions of the theorem to give us the ability to deal with distinct distances between distinct sets.

We denote the projection of $E$ onto its first $k$ coordinates by $\pi_1(E) = \pi_1$, and for the projection of $E$ onto its final $l$ coordinates we use $\pi_2(E) = \pi_2$. We now use dyadic pigeonholing to separate the points in $E$ in two separate cases, based on their `richness' in the first or second projections. For each point $x$ in $E$ we define the $\pi_2$-richness and $\pi_1$-richness respectively as
 
$$R(x) = |\pi_2(\pi_1^{-1}(\{ \pi_1(x) \}))| \hspace{0.5cm} \text{and} \hspace{0.5cm} S(x) = |\pi_2(\pi_1^{-1}(\{ \pi_1(x) \}))|.$$ 

Note that each of these `richness' functions takes a point $x$ in $E$ and associates to it the number of predecessors of $\pi_1(x)$ or $\pi_2(x)$ has in $E$ respectively, thus each function associates the same value to all predecessors of a point in either projection. We note that as $|\pi_1|, |\pi_2| \leq |E|$ we have that $|(S(x)|, |R(x)| \leq |E|$.

We are now ready to apply the dyadic pigeonholing, we detail this for $\pi_2$-richness, but the same process works identically for the $\pi_1$-richness. By Lemma \ref{DyadicPH} we have $\frac{E}{\log(|E|)}$ members of our original $E$ that have some fixed $\pi_2$-richness, say $r$. For emphasis, we call this new subset of $E$ with fixed richness $E'$ were $|E'|=\frac{|E|}{\log(|E|)} \approx |E|$. We note that as our function $R(x)$ is constant on all points mapping to the same point under our $\pi_1$ mapping we do not lose any such points under our pigeonholing, thus each point in $E'$ is one of $\sim r$ that sit above a given member of $\pi_1(E')$ and each point of  $\pi_1(E')$ has $\sim r$ points above it. This means that $|E'| \sim |\pi_1(|E'|)|r$ but as $|E'|\approx |E|$ we have that $|E| \approx |\pi_1(E')|r$.

Doing the same pigeonholing process for $\pi_1$-richness gives us a different set $E''$ of $\frac{|E|}{\log(|E|)}$ points from $E$ with some constant richness $\sim s$ and $|\pi_2(E'')|s \approx |E|$.

To count distances in these let $\gamma_m$ be the best exponent for the Erd\H os distance problem in $\mathbb{R}^m$ and $\delta_m$ be the best exponent for the distance set of two different point sets in $\mathbb{R}^m$ satisfying condition $C_m$. Let us again examine the $\pi_2$-rich case, we have $|\pi_1(E')|$ points in the first projection which is a subset of $\mathbb{R}^k$, hence $|\pi_1(E')|^{\gamma_k}$ distances. We fix one of these distances and fix two endpoint which produce this distance, then as all the points in $\pi_1(E')$ have $\sim r$ predecessors we gain two sets of size $\sim r$ in $\mathbb{R}^l$, thus these generate $\gtrsim r^{\delta_l}$ distances. As this happens for all fixed distances we have $\gtrsim |\pi_1(E')|^{\gamma_k}r^{\delta_l}$ distance pairs.

Via the identical method for $\pi_1$-richness we gain $\gtrsim |\pi_2(E'')|^{\gamma_l}s^{\delta_k}$ distance pairs. combining these two estimates gives

$$ |B_{k,l}(E)| \gtrsim \max \{ |\pi_1(E')|^{\gamma_k}r^{\delta_l}, |\pi_2(E'')|^{\gamma_l}s^{\delta_k} \}. $$

By examining the latter of these two cases we gain that

\begin{equation} \label{Eq: GenBklBound}
 |B_{k,l}(E)| \gtrsim |\pi_2(E'')|^{\gamma_l}s^{\delta_k} \geq \left(|\pi_2(E'')|s\right)^{\min \{ \gamma_l, \delta_k\}} \approx |E|^{\min \{ \gamma_l, \delta_k\}}.
\end{equation}

For the case when $k=l$ we find a $\zeta$ such that $\delta_k = \zeta\gamma_k$.

$$ |B_{k,k}(E)| \gtrapprox  |\pi_1(E')|^{\gamma_k}r^{\delta_k} \geq  \left(|\pi_1(E')|r\right)^{\zeta\gamma_k} \approx |E|^{\zeta\gamma_k}.$$

\subsection*{Proof of Corollary \ref{Coro: Bkl}}

We have $\gamma_m = \frac{2}{m} - \frac{2}{m(m+2)}$ due to Solymosi-Vu (\cite{SV08}) and using Theorem \ref{TPairDistFalconerDiscrete} under the conditions of $s$-adaptability we have $\delta_m = \frac{2}{m+1}$. Using these particular values if $k>l$ it follows that $\min\{ \gamma_l, \delta_k\}=\gamma_l$ and thus (\ref{Eq: GenBklBound}) becomes
$$ |B_{k,l}(E)| \gtrapprox |E|^{\gamma_l}. $$
Note that this is in fact sharp when all of our points of $E$ have there first $k$ coordinates fixed, this necessarily means that $s=1$ in (\ref{Eq: GenBklBound}) and thus this last inequality can be sharp. However, if we were to impose stricter conditions, for example that both projections contain a non-constant proportion of $|E|$ points, we could gain more from the $|B_{k,l}(E)| \gtrsim |\pi_2(E'')|^{\gamma_l}s^{\delta_k}$ inequality.

When $k=l$ we note that the above $\zeta$ has the property that $\zeta = \frac{k^2+2k}{k^2+2k+1}\geq \frac{13}{14}$, this least as $k\geq 3$ (as the case $k=2$ already has an optimal bound above). Using this we can obtain the bound
$$ |B_{k,k}(E)| \gtrapprox |E|^{\frac{13}{14}\gamma_k},$$
thought this can clearly be improved for $k$ larger than $3$.

\vskip.125in 

\section{Proof of Proposition \ref{PContToDiscEnergy}}

For the proof of Proposition \ref{PContToDiscEnergy} we need the following Theorem from a paper of Iosevich, Rudnev and Uriarte-Tuero, which allows us control on the separation of points in an $s$-adaptable point set. See also \cite{I08}. 

\begin{theorem}[Theorem 2.11, \cite{IRU14}]\label{TDisjBalls}
Suppose that $E$ is a $s$-adaptable point set of size $n$, then after rescaling $E$ to the unit cube in $\mathbb{R}^d$, and perhaps removing a set of size at most $\frac{n}{2}$ the minimal separation between two points is $\gtrsim n^{\frac{-1}{s}}$.
\end{theorem}

\subsection*{Proof of Proposition \ref{PContToDiscEnergy}}

We suppose that $E$ is an $s$-adaptable point set. By Theorem \ref{TDisjBalls} we can throw away at most half of our points to gain that for any two points $e$ and $e'$ in $E$, we have that $|e-e'| \gtrsim n^{-\frac{1}{s}}$. For our result it suffices to bound $I_s(\mu_E)$ by $\sum_{e\neq e'}|e-e'|^{-s}$. Using the definition of $\mu_E$ we have that

\begin{align*}
I_s(\mu_E) &= \int \int |x-y|^{-s}d\mu_e(x)d\mu_f(y) \\
				   &= n^{-2}n^{\frac{2d}{s}}\sum_{e\in E, e'\in F} \int \int |x-y|^{-s}\varphi(n^{\frac{1}{s}}(x-e))\varphi(n^{\frac{1}{s}}(y-e'))dxdy\\
				   &=  I + II.
\end{align*}
Where
$$ I = n^{-2}n^{\frac{2d}{s}}\left(\sum_{e=e'} \int_{|x-e|\leq n^{-\frac{1}{s}}} \int_{|y-e'|\leq n^{-\frac{1}{s}}} |x-y|^{-s}dxdy\right), $$
$$ II = n^{-2}n^{\frac{2d}{s}}\left(\sum_{e\neq e'} \int_{|x-e|\leq n^{-\frac{1}{s}}} \int_{|y-e'|\leq n^{-\frac{1}{s}}} |x-y|^{-s}dxdy \right).$$
We bound each of these separately.
\\
\textbf{Case I}.\\
If $e=e'$ we let $x'=x-y$ and $y'=y$, making the change of variables we have

\begin{align*} 
\int_{|x-e|\leq n^{-\frac{1}{s}}} \int_{|y-e'|\leq n^{-\frac{1}{s}}} |x-y|^{-s}dxdy &= \int_{|x'|\leq 2n^{-\frac{1}{s}}} \int_{|y'-e'|\leq n^{-\frac{1}{s}}}|x'|^{-s}dx'dy'\\
  &= n^{-\frac{d}{s}} \int_{|x'|\leq 2n^{-\frac{1}{s}}}|x'|^{-s}dx' \\ 
  &\lesssim n^{-\frac{d}{s}} \int_0^{n^{-\frac{1}{s}}} r^{-s}r^{d-1}dr \\
  &= n^{-\frac{2d}{s}}\cdot n.
\end{align*}  

Thus $I \lesssim n^{-2}n^{\frac{2d}{s}}\left(\sum_{e=e'} n^{-\frac{2d}{s}}\cdot n \right) \sim 1$. We now move to the more interesting case.
\\
\textbf{Case II}.\\
When $e\neq e'$, we have that $||x-y|-|e-e'||< 3n^{-\frac{1}{n}}$. Thus we have that $|x-y| = |e-e'| + \epsilon$ where $\epsilon$ is an error with $|\epsilon|\lesssim n^{-\frac{1}{s}}$. Thus

\begin{equation}
|x-y|^{-s}= (|e-e'|+\epsilon)^{-s} = |e-e'|^{-s}\left(1+\frac{\epsilon}{|e-e'|}\right)^{-s}.
\end{equation}

However, as our set is $s$-separated, we have that $|e-e'|\gtrsim n^{-\frac{1}{s}}$ and thus we have $|x-y|^{s} \lesssim |e-e'|^{-s}$. Applying this estimate to $II$ we have,

\begin{align*}
II &\lesssim  n^{-2}n^{\frac{2d}{s}}\left(\sum_{e\neq e'} \int_{|x-e|\leq n^{-\frac{1}{s}}} \int_{|y-e'|\leq n^{-\frac{1}{s}}} |e-e'|^{-s}dxdy \right) \\
  &= n^{-2}n^{\frac{2d}{s}}\sum_{e\neq e'}|e-e'|^{-s} \cdot n^{-\frac{2d}{s}}\\
  &= n^{-2}\sum_{e\neq e'}|e-e'|^{-s}
\end{align*}

Combining these two cases gives $I_s(\mu_E) \lesssim n^{-2}\sum_{e\neq e'}|e-e'|^{-s}$.

\vspace{0.125in}

\section{Proof of Theorem \ref{TPairDistFalconerDiscrete}}

\vskip.125in

The proof of Theorem \ref{TPairDistFalconerDiscrete} is gained immediately from the following result.

\begin{lemma}\label{LLebMeasEF} Let $E$ and $F$ be two $s$-adaptable point sets both of size $n$ in $\mathbb{R}^m$. If $s > \frac{m}{2}+\frac{1}{3}$ we have
$$ \mathcal{L}(\Delta(E_s,F_s)) > 0,$$
where $\mathcal{L}(A)$ denotes the Lebesgue measure of the set $A$.
\end{lemma}

\begin{proof}[Proof of Theorem \ref{TPairDistFalconerDiscrete}]
We note that for any distance in $\Delta(E,F)$ creates a measure of $\sim n^{-\frac{1}{s}}$ in $\Delta(E_s, F_s)$. So by Lemma \ref{LLebMeasEF}  for $s > \frac{d}{2} + \frac{1}{3}$ we have
$$ 0< C = \mathcal{L}(\Delta(E_s,F_s)) \sim n^{-\frac{1}{s}} |\Delta (E,F)|.$$
Thus setting $s=\frac{d}{2}+\frac{1}{2}$ we gain the bound
$$ |\Delta (E,F)| \gtrsim n^{\frac{2}{d+1}}.$$
\end{proof}

For the proof of Lemma \ref{LLebMeasEF} we follow approch similar to Mattila's developed in \cite{M87}. Consider the following distance density measure $\nu$ for two compactly supported measures $\mu_1$ and $\mu_2$, if it exists, on $\mathbb{R}^m$ by the relation
$$ \int f(t) d\nu(t) = \int \int f(|x-y|) d\mu_1(x) d\mu_2(y). $$
It is easy to see that the measure $\nu$ is finite and has support in $\Delta(E_s,F_s)$. Thus if we let the above measures be $\mu_{E}^s$ and $\mu_{F}^s$ respectively we can use Cauchy-Schwarz to obtain the following
\begin{equation}\label{Ev2Suff}
1 \lesssim \left(\int d\nu(t)\right)^2 \leq \mathcal{L}(\Delta (E_s,F_s)) \cdot \int \nu^2(t) dt.
\end{equation}
Thus it suffices to bound $\int \nu^2(t) dt$. To do so we use two results, the first allows us to write our energy integral in terms of the Fourier transform of our indicator measure on $E$ or $F$. The second bounds this integral in terms or an energy integral.

\begin{lemma}\label{TNuSqBound}Suppose that $\nu$ is a measure defined by $\int f(t)d\nu(t) = \int\int f(|x-y|)d\mu_E(x)d\mu_F(y)$, then there is some constant $c$ such that
$$\int \nu(t)^2 dt \leq \left( \int \left(\int |\widehat{\mu_E}(R\omega)|^2d\omega \right)^2 R^{m-1} dR\right)^{\frac{1}{2}}\left(  \int \left(\int |\widehat{\mu_F}(R\omega)|^2d\omega \right)^2 R^{m-1} dR \right)^{\frac{1}{2}}.$$
\end{lemma}

\begin{proof}
This proof follows the techniques of [Section 5, \cite{GILP15}]. Let $\mathbb{O}(m)$ be the orthogonal group in $\mathbb{R}^m$. Using the proof of Theorem 1.3 from \cite{GILP15} we have 

$$ \int_{\mathbb{R}^m} \nu(t)^2 dt \sim \mu_E \times \mu_F \times \mu_E \times \mu_F (\{ (x,y,x',y') \in \mathbb{R}^4 :  |x-y| \overset{\varepsilon}{=}|x'-y'| \})$$

If $|x-y|=|x'-y'|$ then there are two cases, the first is that the line $x'-y'$ is a translation of $x-y$ or that there is a rotation $\theta$ such that $x-y=\theta(x'-y')$. As Guth-Katz showed in \cite{GK15} we have that the translations do not significantly contribute to this sum and thus it is sufficient to bound those repeated distances associated to rotations. For the details see Section 2 of \cite{GILP15}. Define $v_\theta(z)$ by the integral $\int f(z) d\nu_\theta(z) = \int f(x-\theta y)d\mu_E(x)d\mu_F(y)$ we have

$$ \int_{\mathbb{R}^m} \nu(t)^2 dt \sim \int_{\mathbb{O}(m)}\int_{\mathbb{R}^m}\nu^2_\theta(t)dtd\theta.$$

Thus is is clearly sufficient to bound this second integral.

Setting $f(z)=e^{-2\pi i z\cdot \xi}$ and $g(z)==e^{2\pi i z\cdot \xi}$ yields the identities
$$ \widehat{\nu}_\theta(\xi) = \widehat{\mu}_E(\xi)\overline{\widehat{\mu}_F(\theta^T\xi)} \hspace{0.5cm} \text{and} \hspace{0.5cm} \widehat{\nu}_\theta(\xi) = \overline{\widehat{\mu}_E(\xi)}\widehat{\mu}_F(\theta^T\xi),$$
respectively. So we have $\nu_\theta^2(\xi) = |\widehat{\mu}_E(\xi)|^2|\widehat{\mu}_F(\theta^T\xi)|^2$.
Thus when we integrate $\nu_\theta^2$ over all points in $\mathbb{R}^m$ and over all $\theta$ in $\mathbb{O}(m)$ (recalling $\mu_E$ and $\mu_F$ are compactly supported and thus we can use Frobini's Theorem) we have

\begin{align*}
\int_{\mathbb{O}(m)}\int_{\mathbb{R}^m}\nu^2_\theta(t)dtd\theta &= \int_{\mathbb{O}(m)}\int_{\mathbb{R}^m} \widehat{\nu}_\theta^2(\xi) d\xi\\
			&= \int |\widehat{\mu}_E(\xi)|^2 \left( \int_{\mathbb{O}}|\widehat{\mu}_F(\theta^T\xi)|^2 d\theta \right) d\xi \\
\end{align*}
 
By making a change of variables to the spherical coordinates $(r, \omega)$ in $\mathbb{R}_{\geq 0} \times S^{m-1}$ we have that $\xi = R\omega$ and $\theta \xi = R \omega'$. Thus we have the equality 
$$\int_{\mathbb{O}}|\widehat{\mu}_F(\theta^T\xi)|^2 d\theta = \int_{S^{m-1}}|\widehat{\mu}_F(R\omega')|^2 d\omega'$$
As we are still ranging integrating over all points on a fixed sphere. Using this gives us

$$ \int_{\mathbb{O}(m)}\int_{\mathbb{R}^m}\nu^2_\theta(t)dtd\theta = c \int \left(\int_{S^{m-1}} |\widehat{\mu}_E(R\omega)|^2 d\omega\right) \left(\int_{S^{m-1}}|\widehat{\mu}_F(R\omega')|^2 d\omega' \right) R^{m-1} dR,$$
 
 For some constant $c$ in $\mathbb{R}$. Then distributing the $R^{m-1}$ so that each $\omega$ integral has a $R^{\frac{m-1}{2}}$ factor we have via Cauchy-Schwarz that
 
 $$\int \nu(t)^2 dt \leq \left( \int \left(\int |\widehat{\mu_E}(R\omega)|^2d\omega \right)^2 R^{m-1} dR\right)^{\frac{1}{2}}\left(  \int \left(\int |\widehat{\mu_F}(R\omega)|^2d\omega \right)^2 R^{m-1} dR \right)^{\frac{1}{2}}.$$
\end{proof}

This is an extremely helpful characterisation due to the following Theorem.

\begin{theorem}\label{TmuBound} Let $\mu$ be a compactly supported Borel measure. Then for $s>\frac{d}{2}$,
$$ \int_{S^{d-1}} |\widehat{\mu}(R\omega)|^2d\omega \leq C I_s(\mu)R^{-\beta_s},$$
With $\beta_s = \frac{d+2s-2}{4}$ if $\frac{d}{2}<s\leq \frac{d+2}{2}$, and $\beta_s= s-1$ for $s\geq \frac{d+2}{2}$.
\end{theorem}
For $s\leq\frac{d+2}{2}$ this is due to Wolff \cite{W99} ($d=2$) and Erdo\~ gan \cite{E05} ($d\geq 3$); the easier case of $s\geq \frac{d+2}{2}$ is due to Sj\" olin \cite{S93}. 

\vskip.125in

We are now ready to prove Lemma \ref{LLebMeasEF}.
\begin{proof}[Proof of Lemma \ref{LLebMeasEF}]
From (\ref{Ev2Suff}) it is clear that it suffices to bound $\int \nu^2(t)dt$ were $\nu$ is defined in terms of our measures $\mu_E^s$ and $\mu_F^s$.
Using Lemma \ref{TNuSqBound}, we have that
$$\int \nu(t)^2 dt \leq \left( \int \left(\int |\widehat{\mu_E^s}(R\omega)|^2d\omega \right)^2 R^{m-1} dR\right)^{\frac{1}{2}}\left(  \int \left(\int |\widehat{\mu_F^s}(R\omega)|^2d\omega \right)^2 R^{m-1} dR \right)^{\frac{1}{2}}.$$
From this point the argument runs symmetrically in the $E$ and $F$ components, so we will just focus on showing that the first integral is bounded. We apply Theorem \ref{TmuBound} with $\beta_s = \frac{d+2s-2}{4}$ to one of these factors of $\int |\widehat{\mu_E^s}(R\omega)|^2d\omega$ to gain the following

\begin{align*}
 \int \left(\int |\widehat{\mu_E^s}(R\omega)|^2d\omega \right)^2 R^{m-1} dR &\leq C  \int \left(\int |\widehat{\mu_E^s}(R\omega)|^2d\omega \right) R^{-\beta_s} I_s(\mu_E^s) R^{m-1} dR \\
 				&\leq C I_s(\mu_E^s)  \int |\widehat{\mu_E^s}(\xi)|^2 |\xi|^{-\beta_s}d\xi \\
 				&=  C I_s(\mu_E^s) I_{m+\beta_s}(\mu_E^s)
\end{align*}
If $s>m-\beta_s$ then this final line is bounded by $C\cdot I_s(\mu_E^s)^2$, as our set $E$ is $s$-adaptable this is finite. Rewriting the $s>m-\beta_s$ condition gives $s>\frac{m}{2}+\frac{1}{3}$, which holds by assumption. Thus our integral $\int \nu^2(t)dt$ is bounded and we have our result.
\end{proof}

\vskip.125in

%


\section{Proof of Theorem \ref{T222}}

We consider two cases; the first where a projection is rich in points and our distances come only from this rich plane. In the second case all planes are sparser, however for a fixed distance in a given projection we can ensure that we get lots of combinations with that distance from the other projections. We let $\alpha$ in $\left[\frac{1}{3},1\right]$ be a parameter to be optimised later. Suppose we have a projection with $\gtrsim N^{\alpha}$ points in it, then we have that $B_2^3(E)\gtrsim n^\alpha$.

If this is not the case then we have that all projections have $<n^\alpha$ points in. We consider the points in the first projection, as the other projections are of size at most $n^\alpha$ each, a point in the first projection can have at most $n^{2\alpha}$ predecessors so the average number is at least $n^{(1-2\alpha)}$. Similarly, there must be at least $n^{(1-2\alpha)}$ points in this projection with at least the average number, otherwise we would not have sufficient density to account for all the points in $E$. We call this collection of $n^{(1-2\alpha)}$ points with $n^{(1-2\alpha)}$ predecessors rich points. 

As these rich points lie in $\mathbb{R}^2$ they generate $\gtrapprox n^{(1-2\alpha)}$ distances. We fix one of these distances and fix two rich points which generate this distance, we observe that we now have two disjoint sets of size $n^{(1-2\alpha)}$ in different copies of $\mathbb{R}^4$. It is at this point our $s$-adaptability becomes important, as otherwise we would have no non-trivial bound on $B_{2,2}$ of these sets. However as these are $s$-adaptable we have that these sets generate $n^{(1-2\alpha)\eta_1}$ pairs for each fixed distance in the first coordinate. Giving us $\gtrapprox n^{(1-2\alpha)}\cdot n^{(1-2\alpha)\eta_1}$ distance triples in general. Combining these two bounds gives us
$$ B_2^3(E) \gtrapprox \min \{ n^{\alpha}, n^{(1-2\alpha)(1+\eta_1)} \},$$
which is optimised at $\alpha = \frac{1+\eta_1}{3+2\eta_1}$.

\vskip.125in

\section{Proof of Theorem \ref{T2qTimes}}

We follow closely the argument given of Theorem \ref{T222}. We set up our exponent $\alpha$ in $[\frac{1}{q},1]$ and consider the competing cases of one rich projection against many sparser projections. In the case where the size of the first projection has at least $n^\alpha$ points in we have that $B_2^q(E)$ has at least $n^\alpha$ triples. 

In the second case where all projections have fewer than $n^\alpha$ points, we use density counting to show that any projection has at least $n^{(1-(q-1)\alpha)}$ points with at least $n^{(1-(q-1)\alpha)}$ predecessors, we call such points rich. These rich points generate $\gtrapprox n^{(1-(q-1)\alpha)}$ distances in this projection, we fix one of these distances and look at its endpoints. The predecessors of these two rich points are two sets of size $n^{1-(q-1)\alpha}$ in two disjoint copies of $\mathbb{R}^{2(q-1)}$, which both inherit $s$-adaptability from $E$ as they are both subsets. Thus we can use our exponent for $B_2^{(q-1)}$ to obtain $n^{(1-(q-1)\alpha)\eta_{q-1}}$ distances $(q-1)$-tuples associate to our fixed distance.

As we can do this for all of the $n^{1-(q-1)\alpha}$ distances obtained in a projection we have that $B_2^q(E)$ has $\gtrapprox n^{1-(q-1)\alpha} \cdot n^{(1-(q-1)\alpha)\eta_{q-1}} = n^{(1-(q-1)\alpha)(1+\eta_{q-1})}$ distance $q$-tuples. Combining the estimates from both cases gives us that
$$ B_2^q(E) \gtrapprox \min \{ n^{\alpha}, n^{(1-(q-1)\alpha)(1+\eta_{q-1})} \},$$
which is optimised at $\alpha = \frac{1+\eta_{q-1}}{q+(q-1)\eta_{q-1}}$.

\vskip.125in

\section{Proof of Theorem \ref{TGenPartsDydPideon}}

\vskip.125in

Let $\pi_i$ be the projection of $\mathbb{R}^n$ onto $\mathbb{R}^{p_i}$, we define $R_i(x)=|\pi_i^{-1}(\{ \pi_i(x)\})|$ to be the $p_i$-richness of $x$ in $\mathbb{R}^d$. We can thus partition $E$ with respect to $R_i$ to gain $\frac{|E|}{\log(|E|)}$ points of $E$ with $R_i \sim r_i$ for each such point, call this new set $E_i$. Note that we preserve all points in $E$ with the same $p_i$-richness, thus we have that $|E| \approx|E_i| \sim |\pi(E_i)|r_i$. We now count distances tuples, first as we have $|\pi_i(E_i)|$ points in $\mathbb{R}^{p_i}$ these create at least $|\pi_i(E_i)|^{\gamma_i}$ distances. For each of these distances take a pair of endpoints and look at the points which project onto each, there are $\sim r_i$ of such points that lie in a subset of $\mathbb{R}^{d-p_i}$. Each of these point sets inherit the $s$-adaptability from $E$ and thus we can use the bound for $B_{p\setminus p_i}(A,B)$ to gain $r_i^{\eta_i}$ $(q-1)$-tuples of distances between these two sets. Combining our two distance estimates gives
$$ B_p(E) \gtrsim |\pi_i(E)|^{\gamma_i}r_i^{\eta_i} \gtrapprox |E|^{\min\{\gamma_i, \eta_i\}} $$

\vskip.125in

\section{Proof of Theorem \ref{TGenPart}}

\vskip.125in

We follow closely the argument given of Theorems \ref{T222} and \ref{T2qTimes}. We set up our exponent $\alpha$ in $[\frac{1}{q},1]$ and consider the competing cases of one rich projection against many sparser projections. In the case where the size of the first projection has at least $n^\alpha$ points but as we are in some $\mathbb{R}^{p_i}$ then $B_p(E)$ gains $n^{\gamma_i\alpha}$ q-tuples. It is clear that the worst of these occurs when $i=q$ and thus our first case realises at least $n^{\gamma_q\alpha}$ elements of  $B_p(E)$.

In the second case where all projections have fewer than $n^\alpha$ points, we use density counting to show that any projection has at least $n^{(1-(q-1)\alpha)}$ points with at least $n^{(1-(q-1)\alpha)}$ predecessors, we call such points rich. Rich points in the $i^{th}$ projection generate $\gtrapprox n^{\gamma_i(1-(q-1)\alpha)}$ distances in this projection, we fix one of these distances and look at its endpoints. The predecessors of these two rich points are two sets of size $n^{1-(q-1)\alpha}$ in two disjoint copies of $\mathbb{R}^{d-p_i}$, which both inherit $s$-adaptability from $E$ as they are both subsets. Thus we can use the bound for $s$-adaptable point sets $A$ and $B$ both of size $n^{{1-(q-1)\alpha}}$ to gain that these generate at least $n^{\eta_i(1-(q-1)\alpha)}$ $(q-1)$-tuples. Thus in this second case we have a total of $\gtrapprox n^{\gamma_i(1-(q-1)\alpha)} \cdot n^{(1-(q-1)\alpha)\eta_{i}} = n^{(1-(q-1)\alpha)(\gamma_i+\eta_{i})}$ distance $q$-tuples. Combining the estimates from both cases gives us that
$$ B_p(E) \gtrapprox \min \{ n^{\gamma_q\alpha}, n^{(1-(q-1)\alpha)(\gamma_i+\eta_{i})} \},$$
which is optimised at $\alpha = \frac{\gamma_i+\eta_{i}}{\gamma_q+(q-1)(\gamma_i+\eta_{i})}$. Using this value of $\alpha$ gives us the result

$$ B_p(E) \gtrapprox n^\tau \hspace{0.5cm} \text{where} \hspace{0.5cm} \tau = \gamma_q\left(\frac{\gamma_i +\eta_i}{\gamma_q+(q-1)(\gamma_i+\eta_i)}\right). $$

\vskip.125in

\section{Discussion of Optimal Exponents}

\vskip.125in

There is no reason to believe any of the bounds obtained above to be optimal, indeed the standard example of the integer cube produces gives the following bounds

\begin{example}Suppose that $E_n$ is the integer cube of size $n$ in $\mathbb{R}^d$ with $d\geq 4$ and $p$ an increasing integer partition of $d$ into $q$ integers, then
$$ |B_p(E_n)| \sim n^{2q/d}.$$
Indeed, since in each projection $\pi_i$ into $\mathbb{R}^{p_i}$ there are $n^{\frac{p_i}{d}}$ points coming from members of $E_n$, which form an integer lattice in $\mathbb{R}^{p_i}$. Thus these points create $(n^{\frac{p_i}{d}})^{\frac{2}{p_i}} = n^{\frac{2}{d}}$ distances. For any of these distances, the predecessors of some chosen endpoints will be integer grids in all but the fixed coordinates in the $p_i^{th}$ projection, thus we can recreate any distance in the other coordinates coming from the other projections. In total this gives
$$|B_p(E_n)| \sim \prod_{i=1}^q n^{\frac{2}{d}} \sim n^{2q/d}.$$
\end{example}

A further example of note is were we have a point set $E$ in a $p_q$-dimensional subset of $\mathbb{R}^d$. Suppose that this is done in a way such that the first $d-p_q$ coordinates are fixed, then we have that $B_p(E)$ is just the set of distances of $E$ in this $p_q$-dimensional subset. Thus we have that $B_p(E)\sim E^{\gamma_{p_q}}$. 

In the $q=2$ case initially studied, we have that $|B_{k,l}(E_n)| \sim n^{\frac{4}{k+l}}$, in particular $B_{2,2}(E_n)\sim n$ and $B_{2,3}(E_n)\sim n^{\frac{4}{5}}$. Recall that the bounds achieved in this paper where $1$ for the $(2,2)$-case, $\frac{3}{5}$ for the $(2,3)$-case. Note that the first of these is sharp (up to logarithms) while the second is short of the optimal $\frac{2}{3}$ bound obtained form all of our points lying in a three-dimensional subset of $\mathbb{R}^5$. However, the result we gained in the $(2,3)$-case was really $|B_{2,3}(E)| \gtrsim |E|^{\gamma_3}$ and thus is reliant on progress in the three-dimensional analogue of the Erd\H os-distance problem. An easier question would be to look at situations where your point set was truly five-dimensional, one could impose the condition that no more than $|E|^{\frac{1}{2}}$ of our points of $E$ lie in a subspace of dimension four. Under such conditions one could hope to brake the $|E|^\frac{2}{3}$ barrier and gain a result closer to the $|E|^{\frac{4}{5}}$ obtained by the grid.
 
In the $(k,l)$-case ($k\neq l$) we gain the exponent of $\gamma_l$ and when $k=l$ our exponent is $\frac{k^2+2k}{k^2+2k+1}\gamma_k$. In these cases we had the requirement of $s$-adaptability of our point sets in order to achieve these bounds, this came from the necessity of needing to find distances between two different point sets in higher dimensions. However the need for $s$-adaptability does not appear a necessary requirement for progress of the $B_{k,l}(E)$ bound and thus removing this requirement from the above theorems would be of great interest. In addition this would give hope to removing the additional case when $k=l$, as a different approach may remove the discrepancy between our distance bound on two sets and our distance bound on a single set.

For the partitions of two case the above examples suggests that our aim should be an exponent of one for all dimensions. However the exponent obtained in Theorem \ref{T2qTimes} is only a slight improvement of $\frac{1}{d}$, even using $s$-adaptable sets and the Guth-Katz solution of the Erd\H os-distance problem in the plane. Thus we believe that large improvements are possible for these bounds in particular, although the difficulty gaining `good' bounds here is unclear. As with the general case in the partitions of length $2$, the notion of $s$-adaptability does not seem crucial to the structure of the problem and thus should be able to be removed.


One can also ask for bounds on $B_p$ for a more diverse partitions $p$ hoping to better the exponent $\frac{\gamma_{p_q}}{q}$ obtained in (\ref{ETrivialGenp}). The method used in this paper seems to yield very little in this direction, in particular way one has to deal with distance tuples generated by different sets in high dimensions causes extreme inefficient bounds. Improving the bounds here appears also to be a very tricky proposition and would be of great interest.

\bigskip

\end{document}